\documentclass{amsart}
\usepackage{epsfig}
\usepackage{graphicx}
\usepackage{version}

\newtheorem{theorem}{Theorem}[section]
\newtheorem{lemma}[theorem]{Lemma}
\newtheorem{corollary}[theorem]{Corollary}
\theoremstyle{definition}
\newtheorem{definition}[theorem]{Definition}
\newtheorem{example}[theorem]{Example}

\theoremstyle{remark}
\newtheorem{remark}[theorem]{Remark}

\numberwithin{equation}{section}

\begin{document}

\title{Katugampola fractional integral and fractal dimension of bivariate functions}




\author{S. Verma}
\address{Department of Mathematics, IIT Delhi, New Delhi, India 110016 }
\email{saurabh331146@gmail.com}

\author{P. Viswanathan}
\address{Department of Mathematics, IIT Delhi, New Delhi, India 110016}

\email{viswa@maths.iitd.ac.in}

%    General info
%\subjclass[2000]{Primary 54C40, 14E20; Secondary 46E25, 20C20}

%\date{January 1, 2001 and, in revised form, June 22, 2001.}


\keywords{Fractional integral, Hausdorff dimension, Box dimension, Bounded variation}
\footnote{This work is a part of the Ph.D. thesis of the first author submitted to IIT Delhi. The thesis is defended successfully in October 2020. This paper is submitted to a journal for possible publication.}
\begin{abstract}
The subject of this note is the mixed Katugampola fractional integral of a bivariate function defined on a rectangular region in the Cartesian plane.  This is a natural extension of the Katugampola fractional integral of a univariate function - a  concept well-received in the recent literature on fractional calculus and its applications. It is shown that the mixed Katugampola fractional integral of a prescribed bivariate function preserves  properties such as boundedness, continuity and bounded variation of the function. Furthermore, we estimate fractal dimension of the graph of the mixed Katugampola integral of a continuous bivariate function. Some examples for  bivariate functions that are not of bounded variation but with graphs having box dimension $2$ are constructed.  The findings in the current  note may be viewed as a sequel to our work reported in [Appl. Math. Comp., 339, 2018, pp. 220-230].
\end{abstract}

\maketitle

%%%%%%%%%%%%%%%%%%%%%%%%%%%%%%%%%%%%%%%%%%%%%%%%%%%%%%%%%%%%%%%%%%%%%%%%

%%%%%%%%%%%%%%%%%%%%%%%%%%%%%%%%%%%%%%%%%%%%%%%%%%%%%%%%%%%%%%%%%%%%%%%%
.

\section{INTRODUCTION}
Fractal Geometry (FG) and Fractional Calculus (FC) have become rapidly growing fields  in theory and  applications. FG is the mathematical study of the concepts of self-similarity, fractals, chaos and their applications to the modeling of natural phenomena.
FC is the field of mathematical analysis which deals with the possibility of taking noninteger order differentiation and integration. The present  work lies broadly  in the intersection of these two fields.

\par

Due to its applications in  diverse
 fields of scientific knowledge, FC experienced a fast development  in recent decades. For a short review of various definitions of fractional derivatives and integrals we refer to \cite{OM}, and the book \cite{Samko} for a detailed  treatment.  To the best of our knowledge, the fundamental question that whether it is possible to define a general class of fractional integral/derivative operators that encompass all the standard fractional operators is still open.

\par Two forms of the fractional integral that have been studied extensively for their applications  are the Riemann-Liouville fractional integral and  Hadamard fractional integral. In reference \cite{Katu}, Katugampola provided a new fractional integral of a univariate function which includes both the Riemann-Liouville and Hadamard fractional integrals. See also \cite{KST} for a more general definition of the  fractional integral of a function  with respect to
another function. In the first part of this paper, we extend this to bivariate function and introduce the mixed Katugampola integral for bivariate functions analogous to the mixed Riemann-Liouville integral; see, for instance, \cite{Samko}.

\par The quest for the geometric and physical interpretations of fractional-order operators and linguistical reasons motivated  researchers to explore the
 relation between FC and FG. Although this may not be the primary purpose, in reference \cite{Liang}, Liang obtained the box dimension of the graph of the Riemann-Liouville fractional integral of a continuous function of bounded variation. Results connecting fractional integrals and fractal dimension can be consulted also in \cite{LS,YLS,RSY}.  A counterpart of the result in \cite{Liang} for the Katugampola fractional integral was given later  by the authors in \cite{SV}. The focus in references \cite{Liang,SV} is on the univariate functions and it is not known until now whether the results therein have straightforward extensions to higher dimensions.

 \par
In the second part of this note, we investigate the fractal dimensions 
(Hausdorff dimension and box dimension) of the graph of the mixed Katugampola fractional integral of a continuous bivariate function of bounded variation defined on a rectangle $[a,b] \times [c,d]$ in $\mathbb{R}^2$.
 It is worth to mention that in the context of multivariate functions there are different notions of bounded variation (see, for instance, \cite{James}).  However,  we confine our discussion to  bounded variation in the sense of Arzel\'{a}. While the current note,  as alluded earlier, shares a natural kinship with \cite{SV}, the reader can also recognize a considerable
degree of distinction between them.

 \par We make no claim about the novelty of the techniques used to prove our results. In fact, most of the proofs herein are elementary, rely on  the standard tools in mathematical analysis, and  follow from relevant known results and definitions.
We acknowledge that this paper is a part of the Ph.D. thesis \cite{SVthesis} of the first author submitted to IIT Delhi.

 \section{Background and preliminaries}

 In this section we gather some preliminary definitions and facts that are needed in the sequel.

As hinted in the introductory section, there are several ways to extend the notion of bounded variation to multivariate functions, see \cite{James} for a review of different approaches for functions of two variables. In what follows, we shall recall only the bounded variation of a bivariate function in the sense of Arzel\'{a}.  We shall refer to a function that is not of bounded variation as a function of unbounded variation.
 \begin{definition}\cite{James}
    Let $f :[a,b] \times [c,d] \rightarrow \mathbb{R}$ be a function and $(x_i,y_j) ~~~~(i,j=0,1,2, \dots ,m)$ be a set of points satisfying the conditions $$ a=x_0 \le x_1 \le x_2 \le \dots \le x_m=b;$$
    $$ c =y_0 \le y_1 \le y_2 \le \dots \le y_m=d.$$ Let $ \Delta f(x_i,y_i)=f(x_{i+1},y_{i+1})-f(x_i,y_i).$
    The function $f(x,y)$ is said to be of bounded variation in the sense of Arzel\'{a} if the sum $$ \sum_{i=0}^{m-1} | \Delta f(x_i,y_i)|$$ is bounded for all such sets of points.
\end{definition}

  \begin{theorem}\cite{Raymond} \label{BVK_Arzela}
      A necessary and sufficient condition for a function $f: [a,b] \times [c,d] \rightarrow \mathbb{R}$ is  of bounded variation in Arzel\'{a}'s sense is the following.  The function $f$  can be expressed as the difference between two bounded functions $f_1$ and $f_2$ satisfying the inequalities $$ \Delta_{10} f_i(x,y) \ge 0, \quad  \Delta_{01} f_i(x,y) \ge 0, \quad \quad i = 1,2,$$
     where $ \Delta_{10} f(x_i,y_j)=f(x_{i+1},y_j)-f(x_i,y_j),$
      $ \Delta_{01} f(x_i,y_j)=f(x_i,y_{j+1})-f(x_i,y_j).$
      \end{theorem}
      \begin{remark}
      We shall refer to the  bounded function $f_i$ satisfying conditions in the previous theorem as a monotone function.

      \end{remark}

The two different notions of fractal dimension commonly used are  Hausdorff dimension and box dimension. We refer \cite{Fal} for these definitions and notation.

For a function $f :[a,b] \times [c,d] \rightarrow \mathbb{R}$, let us denote by $G(f)$ its graph defined as
$$G(f) = \big\{(\textbf{x},f(\textbf{x})): \textbf{x} \in [a,b] \times [c,d]  \big\}.$$

 \begin{lemma}\cite{Fal}
 Let $A \subseteq \mathbb{R}^n $ and $f :A \rightarrow \mathbb{R}$ be continuous on $A.$ Then we have $$\dim_H \big(G(f)\big) \ge \dim_H(A).$$
 In particular, if $f :[a,b] \times [c,d] \rightarrow \mathbb{R}$ is continuous, then $\dim_H \big(G(f)\big) \ge 2.$
 \end{lemma}
\begin{definition}
For a function  $f: [a,b] \times [c,d] \rightarrow \mathbb{R}$, the maximum range of $f$ over the rectangle $A:=[a,b] \times [c,d]$ is defined by
 $$ R_f[A]= \sup_{(t,u),(x,y) \in A} |f(t,u)-f(x,y)|.$$
\end{definition}
The following is a bivariate analogue of \textup{\cite[Proposition 11.1]{Fal}}. Proof may be consulted in \cite{SVarx}.
\begin{lemma} \label{Falbiv}
 Let $f :[a,b] \times [c,d] \rightarrow \mathbb{R}$ be continuous. Suppose that $ 0 < \delta < \min\{b-a,d-c\},$ $ \frac{b-a}{\delta} \le m \le 1+ \frac{b-a}{\delta }$ and $ \frac{d-c}{\delta} \le n \le 1+ \frac{d-c}{\delta }$ for some $m ,n \in \mathbb{N}.$ If $N_{\delta}(G(f))$ is the number of $\delta$-cubes that intersect the graph $G(f)$ of the function $f,$ then $$  \sum_{j=1}^{n} \sum_{i=1}^{m} \max\Big\{ \frac{R_f[A_{ij}]}{ \delta},1\Big\}\le N_{\delta}(G(f)) \le 2mn + \frac{1}{\delta} \sum_{j=1}^{n} \sum_{i=1}^{m} R_f[A_{ij}],$$ where $A_{ij}$ denotes the $(i,j)$-th sub-rectangle obtained by the chosen partition of  $[a,b] \times [c,d]$.
 \end{lemma}

 Now the following corollary can be deduced from the above lemma on lines similar to \textup{\cite[Corollary 11.2]{Fal}}; see also \cite{SVarx} .
 \begin{corollary}\label{BVK_holder_dim}
 Let $f:[a,b] \times [c,d] \rightarrow \mathbb{R}$ be a continuous function.

 \begin{enumerate}
 \item Suppose $$\big|f(t,u)-f(x,y)\big| \le c ~~\|(t,u)-(x,y)\|_2^s ~~~~ \forall~~ (t,u),(x, y) \in [a,b] \times [c,d],$$ where $c>0$ and $0 \le s \le 1.$ Then $$\dim_H (G_f)\le  \underline{\dim}_B (G_f)   \le \overline{\dim}_B (G_f) \le 3-s.$$ The conclusion remains true if H\"older condition holds when $\|(t,u)-(x,y)\|_2 < \delta $ for some $\delta >0.$
 \item Suppose that there are numbers $c>0$, $\delta_0 >0$ and $0 \le s \le 1$ with the following property: for each $(x,y) \in [a,b] \times [c,d] $ and $ 0< \delta <\delta_0$ there exists $(t,u)$ such that $\|(t,u)-(x,y)\|_2 \le \delta$ and $$ |f(t,u)-f(x,y)| \ge c \delta^s.$$ Then $ \underline{\dim}_B (G_f) \ge 3 - s.$
\end{enumerate}
 \end{corollary}

\subsection{Fractional integral}
 Two most common forms of fractional integrals that studied for their applications are Riemann-Liouville fractional integral and the Hadamard fractional integral. We shall recall a version of fractional integral introduced by Katugampola \cite{Katu},
which includes these two well known fractional integrals.
\begin{definition}
 The Katugampola fractional integral of $f: [a,b] \to \mathbb{R}$ is defined as $$ \big(^{\rho}_a \mathfrak{I}^{\alpha}f\big)(x)=\frac{(\rho+1)^{1-\alpha}}{\Gamma (\alpha)} \int_a ^x (x^{\rho+1}-t^{\rho+1})^{\alpha-1} t^{\rho}f(t)~\mathrm{d}t,$$ where $\alpha>0$ and $\rho \ne -1$ are real numbers.
\end{definition}
Analogous to the definition of Riemann-Liouville fractional integral for a univariate function, the mixed Riemann-Liouville fractional integral of multivariate function is defined naturally as follows, see, for instance, \cite{Samko}.
\begin{definition}
Let $a=(a_1,a_2,\dots,a_n)$ be a fixed point in $\mathbb{R}^n$ and $\phi(x)= \phi(x_1,x_2,\dots,x_n)$ be a function of $n$ variables given for $x_k>a_k$. The left-hand
sided mixed Riemann-Liouville fractional integral of order $\alpha= (\alpha_1,\alpha_2, \dots, \alpha_n)$ is defined as
$$ \big({\mathfrak{R}}_{a}^{\alpha}\phi \big)(x) = \frac{1}{\Gamma (\alpha)}\int_{a_1}^{x_1} \dots \int_{a_n} ^{x_n} \frac{\phi(t)}{(x-t)^{1-\alpha}}~ \mathrm{d}t, \quad \alpha>0,$$
where $\Gamma (\alpha) = \Gamma (\alpha_1) \Gamma (\alpha_2) \dots \Gamma(\alpha_n),$ $\mathrm{d}t=\mathrm{d}t_1\dots \mathrm{d}t_n$ and $(x-t)^{\alpha-1}=(x_1-t_1)^{\alpha_1-1} \dots (x_n-t_n)^{\alpha_n-1}$.
In particular, for a function $f$ defined  on a closed rectangle $[a,b] \times [c,d],$ the left-hand sided mixed Riemann-Liouville fractional integral of $f$ is defined as $$ \big({\mathfrak{R}}_{(a,c)}^{(\alpha_1,\alpha_2)}f \big)(x,y)=\frac{1}{\Gamma (\alpha_1)  \Gamma (\alpha_2)} \int_a ^x \int_c ^y (x-t_1)^{\alpha_1-1} (y-t_2)^{\alpha_2-1}f(t_1,t_2)~\mathrm{d}t_1 ~ \mathrm{d}t_2 ,$$ where $\alpha_1 >0 , \alpha_2 >0.$
Similarly, the right-hand sided mixed fractional integral can be defined.
\end{definition}

%-------------------------------------------------------------------------------------------------------------------------------------------

\section{The Mixed Katugampola fractional integral and some basic properties}
Let  $ 0 < a <b <\infty$ and $ 0 < c < d <\infty.$ On lines similar to the mixed Riemann-Liouville fractional integral, in what follows, we define the notion of mixed Katugampola fractional integral of a bivariate function.
\begin{definition}
    Let $f$ be a function on a closed rectangle $[a,b] \times [c,d].$ The mixed (left-hand sided) Katugampola fractional integral of $f$ is defined as
    \begin{equation*}
    \begin{split}
     \big(^{(p,q)}_{(a,c)} \mathfrak{I}^{(\alpha, \beta)}f\big)(x,y)= &~\frac{(p+1)^{1-\alpha}(q+1)^{1-\beta}}{\Gamma (\alpha) \Gamma (\beta) }.\\ &~ \int_a ^x \int_c ^y (x^{p+1}-s^{p+1})^{\alpha-1}  (y^{q+1}-t^{q+1})^{\beta-1} s^{p} t^{q}f(s,t)~\mathrm{d}s ~ \mathrm{d}t,
     \end{split}
     \end{equation*}
     where $\alpha, \beta > 0$ and $(p,q) \ne (-1,-1).$
    \end{definition}
\begin{remark}
When $(p,q) =(0,0),$ one reobtains the mixed Riemann-Liouville fractional integral, as expected. Moreover, using the L'hospital rule, when $p,q \rightarrow -1^+,$ we have,
   \begin{equation*}
       \begin{aligned}
        &\lim_{p,q \rightarrow -1^+} ~~ \big(^{(p,q)}_{(a,c)} \mathfrak{I}^{(\alpha, \beta)}f\big)(x,y)\\=&\lim_{p,q \rightarrow -1^+} \frac{(p+1)^{1-\alpha}(q+1)^{1-\beta}}{\Gamma (\alpha)\Gamma (\beta)} \int_a ^x \int_c ^y (x^{p+1}-s^{p+1})^{\alpha-1}  (y^{q+1}-t^{q+1})^{\beta-1} s^{p} t^{q}f(s,t)~\mathrm{d}s ~ \mathrm{d}t\\
        =& ~\frac{1}{\Gamma (\alpha) \Gamma (\beta)} \int_a ^x \int_c^y \lim_{p \rightarrow -1^+} \Big(\frac{x^{p+1}-s^{p+1}}{\rho+1}\Big)^{\alpha-1} s^{p} \lim_{q \rightarrow -1^+} \Big(\frac{y^{q+1}-t^{q+1}}{q+1}\Big)^{\beta-1} t^{q}f(s,t)~\mathrm{d}s ~ \mathrm{d}t\\
        =& ~\frac{1}{\Gamma (\alpha) \Gamma (\beta)} \int_a ^x \int_c^y (\log \frac{x}{s})^{\alpha-1} (\log \frac{y}{t})^{\beta-1} \frac{f(s,t)}{st}~\mathrm{d}s ~ \mathrm{d}t,
        \end{aligned}
                  \end{equation*}
   which coincides with the mixed Hadamard fractional integral.
   \end{remark}

\begin{theorem}

If $f$ is bounded function on $[a,b] \times [c,d] $, $\alpha >0$, $\beta >0$ and $p,q >-1$, then $^{(p,q)}_{(a,c)} \mathfrak{I}^{(\alpha, \beta)}f$
 is bounded.
\end{theorem}
\begin{proof}
Since $f$ is bounded, there exists $M>0$ such that $$|f(x,y)|\le M ~~\forall~~ (x,y) \in [a,b] \times [c,d].$$ Now,
\begin{equation*}
    \begin{aligned}
     &~\Big| \big(^{(p,q)}_{(a,c)} \mathfrak{I}^{(\alpha, \beta)}f \big)(x,y)     \Big|\\ ~= &\Bigg|\frac{(p+1)^{1-\alpha}(q+1)^{1-\beta}}{\Gamma (\alpha) \Gamma (\beta) }~\int_a ^x \int_c ^y (x^{p+1}-s^{p+1})^{\alpha-1}  (y^{q+1}-t^{q+1})^{\beta-1} s^{p} t^{q}f(s,t)~\mathrm{d}s ~ \mathrm{d}t\Bigg|\\
     \le & ~\frac{(p+1)^{1-\alpha}(q+1)^{1-\beta}}{\Gamma (\alpha) \Gamma (\beta) } \int_a ^x \int_c ^y \Big|(x^{p+1}-s^{p+1})^{\alpha-1}  (y^{q+1}-t^{q+1})^{\beta-1} s^{p} t^{q} \Big| \Big|f(s,t)\Big| ~\mathrm{d}s ~ \mathrm{d}t\\
     \le&~ \frac{M(p+1)^{1-\alpha}(q+1)^{1-\beta}}{\Gamma (\alpha) \Gamma (\beta) } \int_a ^x \int_c ^y (x^{p+1}-s^{p+1})^{\alpha-1}  (y^{q+1}-t^{q+1})^{\beta-1} s^{p} t^{q} ~\mathrm{d}s ~ \mathrm{d}t.
     \end{aligned}
               \end{equation*}
For $ 0<a<b<\infty $ and $  p,q >-1$, we have positive values inside the modulus in right side of the above inequality. Applying the transformation  $ s^{p+1}=u,t^{q+1}=v,$ in the double integral in the right hand side,  we get
\begin{equation*}
    \begin{aligned}
   &~ \Big| \big(^{(p,q)}_{(a,c)} \mathfrak{I}^{(\alpha, \beta)}f \big)(x,y)     \Big|\\ \le &~  \frac{M(p+1)^{-\alpha}(q+1)^{-\beta}}{\Gamma (\alpha) \Gamma (\beta)} \int_{a^{p+1}} ^{x^{p+1}}  \int_{c^{q+1}} ^{y^{q+1}} (x^{p+1}-u)^{\alpha-1}(y^{q+1}-v)^{\beta-1} ~\mathrm{d}u ~ \mathrm{d}v\\
      \le &~  \frac{M(p+1)^{-\alpha}(q+1)^{-\beta}}{\Gamma (\alpha+1) \Gamma (\beta +1) } (b^{p+1}-a^{p+1})^{\alpha} (d^{q+1}-c^{q+1})^{\beta},
     \end{aligned}
               \end{equation*}
  which reveals that $^{(p,q)}_{(a,c)} \mathfrak{I}^{(\alpha, \beta)}f$ is bounded.
\end{proof}

\begin{theorem} \label{BVK_Conti}
Let   $ p,q >-1$ and $\alpha, \beta >0.$ The mixed Katugampola fractional integral  $^{(p,q)}_{(a,c)} \mathfrak{I}^{(\alpha, \beta)}f$ of a 
 continuous function  $f: [a,b] \times [c,d] \to \mathbb{R}$,   is continuousl.
\end{theorem}
\begin{proof}
Let $ 0<a\le x <x+h\le b $ and $ 0<c \le y < y+k \le d.$
\small
\begin{equation*}
    \begin{aligned}
    & \big(^{(p,q)}_{(a,c)} \mathfrak{I}^{(\alpha, \beta)}f\big)(x+h,y+k)-\big(^{(p,q)}_{(a,c)} \mathfrak{I}^{(\alpha, \beta)}f\big) (x,y)\\=& \frac{(p+1)^{1-\alpha}(q+1)^{1-\beta}}{\Gamma (\alpha) \Gamma (\beta) } \Bigg[\int_a ^{x+h} \int_c ^{y+k} ((x+h)^{p+1}-s^{p+1})^{\alpha-1}  ((y+k)^{q+1}-t^{q+1})^{\beta-1} \\ &.s^{p} t^{q}f(s,t)~\mathrm{d}s ~ \mathrm{d}t- \int_a ^x \int_c ^y (x^{p+1}-s^{p+1})^{\alpha-1}  (y^{q+1}-t^{q+1})^{\beta-1} s^{p} t^{q}f(s,t)~\mathrm{d}s ~ \mathrm{d}t \Bigg]\\
      =& \frac{(p+1)^{1-\alpha}(q+1)^{1-\beta}}{\Gamma (\alpha) \Gamma (\beta) } \Bigg[\int_a ^{a+h} \int_c ^{c+k}((x+h)^{p+1}-s^{p+1})^{\alpha-1}  ((y+k)^{q+1}-t^{q+1})^{\beta-1} \\ &.s^{p} t^{q}f(s,t)~\mathrm{d}s ~ \mathrm{d}t+ \int_{a+h} ^{x+h} \int_c ^{c+k}((x+h)^{p+1}-s^{p+1})^{\alpha-1}  ((y+k)^{q+1}-t^{q+1})^{\beta-1} \\ &.s^{p} t^{q}f(s,t)~\mathrm{d}s ~ \mathrm{d}t+  \int_a ^{a+h} \int_{c+k} ^{y+k}((x+h)^{p+1}-s^{p+1})^{\alpha-1}  ((y+k)^{q+1}-t^{q+1})^{\beta-1}\\ &. s^{p} t^{q}f(s,t)~\mathrm{d}s ~ \mathrm{d}t + \int_{a+h} ^{x+h} \int_{c+k} ^{y+k} ((x+h)^{p+1}-s^{p+1})^{\alpha-1}  ((y+k)^{q+1}-t^{q+1})^{\beta-1} \\ &.s^{p} t^{q}f(s,t)~\mathrm{d}s ~ \mathrm{d}t
      - \int_a ^x \int_c ^y (x^{p+1}-s^{p+1})^{\alpha-1}  (y^{q+1}-t^{q+1})^{\beta-1} s^{p} t^{q}f(s,t)~\mathrm{d}s ~ \mathrm{d}t \Bigg]\\
 := & \frac{(p+1)^{1-\alpha}(q+1)^{1-\beta}}{\Gamma (\alpha) \Gamma (\beta) }\big[I_1 + I_2+I_3+I_4 - I_5\big].
           \end{aligned}
               \end{equation*}
    \normalsize
 Let us consider the integral:

 $$I_4:=   \int_{a+h} ^{x+h} \int_{c+k} ^{y+k} ((x+h)^{p+1}-s^{p+1})^{\alpha-1}  ((y+k)^{q+1}-t^{q+1})^{\beta-1} s^{p} t^{q}f(s,t)~\mathrm{d}s ~ \mathrm{d}t.$$

 Consider the transformation  $(x+h)^{p+1}-s^{p+1}=x^{p+1}-u^{p+1}$. Noting that  $a_*:=\sqrt[p+1]{(x^{p+1}-(x+h)^{p+1}+(a+h)^{p+1})} \le a,$ we have
  \small
  \begin{equation*}
      \begin{aligned}
       I_4 =&~  \int_{a_*} ^{x} \int_{c+k} ^{y+k} (x^{p+1}-u^{p+1})^{\alpha-1}  ((y+k)^{q+1}-t^{q+1})^{\beta-1} u^{p} t^{q} \\& .f(\sqrt[p+1]{(u^{p+1}+(x+h)^{p+1}-x^{p+1})},t)~\mathrm{d}u ~ \mathrm{d}t\\
       :=&~ I_{4,1} + I_{4,2},
  \end{aligned}
                 \end{equation*}
                    \normalsize
where
 \small
\begin{equation*}
      \begin{aligned}
       I_{4,1} =&~  \int_{a_*} ^{a} \int_{c+k} ^{y+k} (x^{p+1}-u^{p+1})^{\alpha-1}  ((y+k)^{q+1}-t^{q+1})^{\beta-1} u^{p} t^{q} \\&.f(\sqrt[p+1]{(u^{p+1}+(x+h)^{p+1}-x^{p+1})},t)~\mathrm{d}u ~ \mathrm{d}t\\
      I_{4,2} = &~  \int_{a} ^{x} \int_{c+k} ^{y+k} (x^{p+1}-u^{p+1})^{\alpha-1}  ((y+k)^{q+1}-t^{q+1})^{\beta-1} u^{p} t^{q} \\ &.f(\sqrt[p+1]{(u^{p+1}+(x+h)^{p+1}-x^{p+1})},t)~\mathrm{d}u ~ \mathrm{d}t\
  \end{aligned}
                 \end{equation*}

 \normalsize
Turning to the integral $I_{4,2}$, we see from the substitution $(y+k)^{q+1}-t^{q+1}=y^{q+1}-v^{q+1}$ that

$$I_{4,2} = I_{4,2}^1 + I_{4,2}^2,$$
 where
 \begin{equation*}
 \begin{split}
 I_{4,2}^1 =&~  \int_{a} ^{x} \int_{\sqrt[q+1]{(y^{q+1}-(y+k)^{q+1}+(c+k)^{q+1})}} ^{c} (x^{p+1}-u^{p+1})^{\alpha-1}  (y^{q+1}-v^{q+1})^{\beta-1} u^{p} v^{q}\\ & .f\Big(\sqrt[p+1]{(u^{p+1}+(x+h)^{p+1}-x^{p+1})},\sqrt[q+1]{(v^{q+1}+(y+k)^{q+1}-y^{q+1})}\Big)~\mathrm{d}u ~ \mathrm{d}v\\
  I_{4,2}^2=&~  \int_a ^x \int_c ^y (x^{p+1}-u^{p+1})^{\alpha-1}  (y^{q+1}-v^{q+1})^{\beta-1} u^{p} v^{q}\\ & .\Big[f\Big(\sqrt[p+1]{(u^{p+1}+(x+h)^{p+1}-x^{p+1})},\sqrt[q+1]{(v^{q+1}+(y+k)^{q+1}-y^{q+1})}\Big) \mathrm{d}u ~ \mathrm{d}v
 \end{split}
 \end{equation*}
 Combining these we obtain
\begin{equation*}
       \begin{aligned}
        &\big(^{(p,q)}_{(a,c)} \mathfrak{I}^{(\alpha, \beta)}f\big)(x+h,y+k)-\big(^{(p,q)}_{(a,c)} \mathfrak{I}^{(\alpha, \beta)}f\big) (x,y)\\
        &= \frac{(p+1)^{1-\alpha}(q+1)^{1-\beta}}{\Gamma (\alpha) \Gamma (\beta)  } \Bigg[\int_a ^{a+h} \int_c ^{c+k}((x+h)^{p+1}-s^{p+1})^{\alpha-1}  ((y+k)^{q+1}-t^{q+1})^{\beta-1} \\ & .s^{p} t^{q}f(s,t)~\mathrm{d}s ~ \mathrm{d}t+         \int_{a+h} ^{x+h} \int_c ^{c+k}((x+h)^{p+1}-s^{p+1})^{\alpha-1}  ((y+k)^{q+1}-t^{q+1})^{\beta-1} \\ & .s^{p} t^{q}f(s,t)~\mathrm{d}s ~ \mathrm{d}t +         \int_a ^{a+h} \int_{c+k} ^{y+k}((x+h)^{p+1}-s^{p+1})^{\alpha-1}  ((y+k)^{q+1}-t^{q+1})^{\beta-1}\\ & . s^{p} t^{q}f(s,t)~\mathrm{d}s ~ \mathrm{d}t+        \int_{ \sqrt[p+1]{x^{p+1}-(x+h)^{p+1}+(a+h)^{p+1}}} ^{a} \int_{c+k} ^{y+k} (x^{p+1}-s^{p+1})^{\alpha-1}  \\&.((y+k)^{q+1}-t^{q+1})^{\beta-1} s^{p} t^{q} f(\sqrt[p+1]{s^{p+1}+(x+h)^{p+1}-x^{p+1}},t)~\mathrm{d}s ~ \mathrm{d}t\\
         &+        \int_{a} ^{x} \int_{\sqrt[q+1]{y^{q+1}-(y+k)^{q+1}+(c+k)^{q+1}}} ^{c} (x^{p+1}-s^{p+1})^{\alpha-1}  (y^{q+1}-t^{q+1})^{\beta-1} s^{p} t^{q}\\ &. f\Big(\sqrt[p+1]{s^{p+1}+(x+h)^{p+1}-x^{p+1}},\sqrt[q+1]{t^{q+1}+(y+k)^{q+1}-y^{q+1}}\Big)~\mathrm{d}s ~ \mathrm{d}t\\
               &+
                     \int_a ^x \int_c ^y (x^{p+1}-s^{p+1})^{\alpha-1}  (y^{q+1}-t^{q+1})^{\beta-1} s^{p} t^{q}\\ & .\Big[f\Big(\sqrt[p+1]{s^{p+1}+(x+h)^{p+1}-x^{p+1}},\sqrt[q+1]{t^{q+1}+(y+k)^{q+1}-y^{q+1}}\Big)- f(s,t)\Big]~\mathrm{d}s ~ \mathrm{d}t \Bigg].
        \end{aligned}
                  \end{equation*}
 By the triangle inequality it is plain to see that
  \begin{equation*}
        \begin{aligned}
         &\Big|\big(^{(p,q)}_{(a,c)} \mathfrak{I}^{(\alpha, \beta)}f\big)(x+h,y+k)-\big(^{(p,q)}_{(a,c)} \mathfrak{I}^{(\alpha, \beta)}f\big) (x,y)\Big|\\
          &\le  \frac{(p+1)^{1-\alpha}(q+1)^{1-\beta}}{\Gamma (\alpha) \Gamma (\beta)  } \\&  \Bigg[\int_a ^{a+h} \int_c ^{c+k}|((x+h)^{p+1}-s^{p+1})^{\alpha-1}  ((y+k)^{q+1}-t^{q+1})^{\beta-1} s^{p} t^{q}f(s,t)|~\mathrm{d}s ~ \mathrm{d}t\\ &+         \int_{a+h} ^{x+h} \int_c ^{c+k}|((x+h)^{p+1}-s^{p+1})^{\alpha-1}  ((y+k)^{q+1}-t^{q+1})^{\beta-1} s^{p} t^{q}f(s,t)|~\mathrm{d}s ~ \mathrm{d}t\\ &+         \int_a ^{a+h} \int_{c+k} ^{y+k}|((x+h)^{p+1}-s^{p+1})^{\alpha-1}  ((y+k)^{q+1}-t^{q+1})^{\beta-1} s^{p} t^{q}f(s,t)|~\mathrm{d}s ~ \mathrm{d}t \\ &+        \int_{ \sqrt[p+1]{x^{p+1}-(x+h)^{p+1}+(a+h)^{p+1}}} ^{a} \int_{c+k} ^{y+k} |(x^{p+1}-s^{p+1})^{\alpha-1}  ((y+k)^{q+1}-t^{q+1})^{\beta-1} \\& .s^{p} t^{q} f(\sqrt[p+1]{s^{p+1}+(x+h)^{p+1}-x^{p+1}},t)|~\mathrm{d}s ~ \mathrm{d}t\\
          &+       \int_{a} ^{x} \int_{\sqrt[q+1]{y^{q+1}-(y+k)^{q+1}+(c+k)^{q+1}}} ^{c} |(x^{p+1}-s^{p+1})^{\alpha-1}  (y^{q+1}-t^{q+1})^{\beta-1} \\& .s^{p} t^{q} f(\sqrt[p+1]{s^{p+1}+(x+h)^{p+1}-x^{p+1}},\sqrt[q+1]{t^{q+1}+(y+k)^{q+1}-y^{q+1}})|~\mathrm{d}s ~ \mathrm{d}t\\
                &+
                      \int_a ^x \int_c ^y |(x^{p+1}-s^{p+1})^{\alpha-1}  (y^{q+1}-t^{q+1})^{\beta-1} s^{p} t^{q}|\\& .|f(\sqrt[p+1]{s^{p+1}+(x+h)^{p+1}-x^{p+1}},\sqrt[q+1]{t^{q+1}+(y+k)^{q+1}-y^{q+1}})- f(s,t)|~\mathrm{d}s ~ \mathrm{d}t\Bigg].
         \end{aligned}
                   \end{equation*}
 By the continuity of  $f$, we have: \\ (i)  $|f(t,u)| \le M$ for every $(t,u) \in [a,b]\times [c,d] $ and for some  $M>0$.\\  (ii)  for a  given $\epsilon >0,$ there exists $ \delta > 0 $ such that $|f(t,u)-f(z,w)| < \epsilon $ whenever $\|(t,u)-(z,w)\|_2 < \delta.$\\
 For suitable constants $M_i$, $i=1,2\dots,6$ one can deduce that
 \begin{equation*}
        \begin{aligned}
         &\Big|\big(^{(p,q)}_{(a,c)} \mathfrak{I}^{(\alpha, \beta)}f\big)(x+h,y+k)-\big(^{(p,q)}_{(a,c)} \mathfrak{I}^{(\alpha, \beta)}f\big) (x,y)\Big|\\
          \le & ~ M_1~\int_a ^{a+h} \int_c ^{c+k} ~\mathrm{d}s ~ \mathrm{d}t
          +  M_2~\int_{a+h} ^{x+h} \int_c ^{c+k}((x+h)^{p+1}-s^{p+1})^{\alpha-1}   s^{p} ~\mathrm{d}s ~ \mathrm{d}t\\ &+        M_3 \int_a ^{a+h} \int_{c+k} ^{y+k}  ((y+k)^{q+1}-t^{q+1})^{\beta-1}  t^{q}~\mathrm{d}s ~ \mathrm{d}t \\ &+      M_4  \int_{ \sqrt[p+1]{x^{\rho+1}-(x+h)^{\rho+1}+(a+h)^{\rho+1}}} ^{a} \int_{c+k} ^{y+k}   ((y+k)^{q+1}-t^{q+1})^{\beta-1}  t^{q} ~\mathrm{d}s ~ \mathrm{d}t\\
                    &+      M_5  \int_{a} ^{x} \int_{\sqrt[q+1]{y^{q+1}-(y+k)^{q+1}+(c+k)^{q+1}}} ^{c} (x^{p+1}-s^{p+1})^{\alpha-1}   s^{p} ~\mathrm{d}s ~ \mathrm{d}t\\
                          &+
                                M_6~ \epsilon~ \int_a ^x \int_c ^y (x^{p+1}-s^{p+1})^{\alpha-1}  (y^{q+1}-t^{q+1})^{\beta-1} s^{p} t^{q}~\mathrm{d}s ~ \mathrm{d}t.
         \end{aligned}
                   \end{equation*}
 Further calculations yield
\begin{equation*}
        \begin{aligned}
         &\Big|\big(^{(p,q)}_{(a,c)} \mathfrak{I}^{(\alpha, \beta)}f\big)(x+h,y+k)-\big(^{(p,q)}_{(a,c)} \mathfrak{I}^{(\alpha, \beta)}f\big) (x,y)\Big|\\
          \le &~  M_1 hk +  \bar{M_2}k +        \bar{M_3} h + \bar{M_6}~ \epsilon \\ &+      \bar{M_4} (a- \sqrt[p+1]{x^{p+1}-(x+h)^{p+1}+(a+h)^{p+1}} \\
                    &+     \bar{ M_5}   (c- \sqrt[q+1]{y^{q+1}-(y+k)^{q+1}+(c+k)^{q+1}}.
             \end{aligned}
         \end{equation*}
From the above estimate we infer that $^{(p,q)}_{(a,c)} \mathfrak{I}^{(\alpha, \beta)}f$  is continuous.
\end{proof}
Let us recall the notation  $$ \Delta_{10} f(x_i,y_j)=f(x_{i+1},y_j)-f(x_i,y_j), \quad
       \Delta_{01} f(x_i,y_j)=f(x_i,y_{j+1})-f(x_i,y_j).$$
Let us call  $f$  to be  monotone if $ \triangle_{10} f \ge 0$ and $ \triangle_{01} f \ge 0.$ This next lemma can be viewed as the bivariate analogue of \textup{\cite[Theorem 6.6]{ZW}}; see  also \textup{\cite[Lemma 2.2]{Liang}}. For a proof,   we refer to \cite{SVarx}.
\begin{lemma}\label{Katulem}
If $f:[a,b] \times [c,d] \rightarrow \mathbb{R} $ is of bounded variation in the sense of Arzel\'a, then the following holds:
\begin{enumerate}

\item  If $f(a,c) \ge 0,$ then there exist monotone functions $f_1$ and $f_2$ such that $f=f_1 -f_2$ with $f_1(a,c) \ge 0$ and $f_2(a,c)=0.$
\item  If $f(a,c) < 0,$ then there exist monotone function $f_1$ and $f_2$ such that $f=f_1 - f_2$ with $f_1(a,c) = 0$ and $f_2(a,c) >0.$

\end{enumerate}
\end{lemma}

\begin{theorem}\label{BVK_boundedvar}
Let   $ p,q >-1$ and $\alpha, \beta >0.$
If $f$ is of bounded variation on $[a,b] \times [c,d]$ in the sense of Arzel\'a, then the mixed Katugampola integral
$^{(p,q)}_{(a,c)} \mathfrak{I}^{(\alpha, \beta)}f$   is also of bounded variation on $[a,b] \times [c,d].$
\end{theorem}
\begin{proof}
Since $f$ is of bounded variation on $[a,b] \times [c,d]$ in the sense of Arzel\'a, by Lemma \ref{Katulem} it follows that $$f(x,y)=f_1(x,y)-f_2(x,y) ~~~~\forall ~~(x,y) \in [a,b] \times [c,d],$$ where $f_1$ and $f_2$ are monotone functions. It is enough to show that $^{(p,q)}_{(a,c)} \mathfrak{I}^{(\alpha, \beta)}f$ is a difference of two monotone functions.\par First let us assume $f(a,c) \ge 0. $  By Lemma \ref{Katulem}, we can choose $f_1(a,c) \ge 0$ and $f_2(a,c) = 0.$ Define functions $F_1$ and $F_2$ as follows.
$$F_1(x,y):=  \big(^{(p,q)}_{(a,c)}\mathfrak{I}^{(\alpha, \beta)}f_1\big)(x,y), \quad F_2(x,y):= \big(^{(p,q)}_{(a,c)}\mathfrak{I}^{(\alpha, \beta)}f_1\big)(x,y).$$
By the linearity of the mixed Katugampola fractional integral  $$\big(^{(p,q)}_{(a,c)}\mathfrak{I}^{(\alpha, \beta)}f\big)(x,y)=F_1(x,y)-F_2(x,y).$$
 Now we show that $F_1$ and $F_2$ are monotone functions. To this end, let $a\le x_1 \le x_2 \le b$ and $ c \le y \le d.$
 \small
\begin{equation*}
    \begin{aligned}
     F_1(x_2,y)-F_1(x_1,y)=& ~ \frac{(p+1)^{1-\alpha}(q+1)^{1-\beta}}{\Gamma (\alpha) \Gamma (\beta) }  \int_a ^{x_2} \int_c ^{y} (x_2^{p+1}-s^{p+1})^{\alpha-1} (y^{q+1}-t^{q+1})^{\beta-1}\\&  s^p t^q f_1(s,t)~\mathrm{d}s~\mathrm{d}t - \frac{(p+1)^{1-\alpha}(q+1)^{1-\beta}}{\Gamma (\alpha) \Gamma (\beta) }  \int_a ^{x_1} \int_c ^y (x_1^{p+1}-s^{p+1})^{\alpha-1}\\& (y^{q+1}-t^{q+1})^{\beta-1} s^p t^q f_1(s,t)~\mathrm{d}s~\mathrm{d}t.
           \end{aligned}
               \end{equation*}
               \normalsize
Putting $x_1^{p+1}-s^{p+1}=x_2^{p+1}-u^{p+1}$ in the second term, we have $s^{p}~\mathrm{d}s=u^{p}~\mathrm{d}u$. The limits of integration are $u=\sqrt[p+1]{x_2^{p+1}-x_1^{p+1}+a^{p+1}}$ and $u=x_2.$ Therefore, the second term in the right hand side of the above equation takes the following form
\begin{equation*}
\begin{split}
\frac{(p+1)^{1-\alpha}(q+1)^{1-\beta}}{\Gamma (\alpha) \Gamma (\beta)  } \int_{\sqrt[p+1]{x_2^{p+1}-x_1^{p+1}+a^{p+1}}} ^{x_2} \int_c ^y 
(x_2^{p+1}-u^{p+1})^{\alpha-1} & \\ .(y^{q+1}-t^{q+1})^{\beta-1}      u^{p} t^q
f_1(\sqrt[p+1]{u^{p+1}+x_1^{p+1}-x_2^{p+1}},t)~\mathrm{d}u~\mathrm{d}t.
\end{split}
\end{equation*}
Consequently,
\small
\begin{equation*}
    \begin{aligned}
    & F_1(x_2,y)-F_1(x_1,y)\\
              =& ~ \frac{(p+1)^{1-\alpha}(q+1)^{1-\beta}}{\Gamma (\alpha) \Gamma (\beta)  } \Bigg[ \int_a ^{\sqrt[p+1]{{x_2^{p+1}-x_1^{p+1}+a^{p+1}}}} \int_c ^{y} (x_2^{p+1}-s^{p+1})^{\alpha-1} (y^{q+1}-t^{q+1})^{\beta-1} \\ & .s^p t^q f_1(s,t)~\mathrm{d}s~\mathrm{d}t + \int_{\sqrt[p+1]{x_2^{p+1}-x_1^{p+1}+a^{p+1}}} ^{x_2} \int_c ^y (x_2^{p+1}-s^{p+1})^{\alpha-1} (y^{q+1}-t^{q+1})^{\beta-1}  s^p t^q \\
        & .\Big(f_1(s,t)  - f_1(\sqrt[p+1]{s^{p+1}+x_1^{p+1}-x_2^{p+1}},t)\Big)~\mathrm{d}s~\mathrm{d}t\Bigg].
                   \end{aligned}
               \end{equation*}
               \normalsize
Since $ \sqrt[p+1]{s^{p+1}+x_1^{p+1}-x_2^{p+1} }\le s, $ $ f_1(a,c)\ge 0$ and $f_1$ is increasing with respect to the first variable, it follows that  $F_1(x_2,y)-F_1(x_1,y) \ge 0,$ that is, $\triangle_{10}F_1 \ge 0.$ On similar lines, for $ c \le y_1 \le y_2 \le d$ and $a \le x \le b,$ we arrive at $\triangle_{01}F_1 \ge 0.$ Therefore, $F_1$ is monotone. Similarly, $F_2$ is  monotone as well.\\
If $f(a,c)<0,$  by  Lemma \ref{Katulem} we choose  $f_1$ and $f_2$ such that  $f_1(a,c)=0$ and $f_2(a,c)>0.$
As in the previous case, we deduce that $F_1$ and $F_2$ are monotone functions, completing the proof.
\end{proof}
Here is  a preparatory lemma which may be viewed as a bivariate analogue of the main theorem in Liang \textup{\cite[Theorem 1.5]{Liang}}.
\begin{lemma}\label{BVK_main}
If $f :[a,b] \times [c,d] \rightarrow \mathbb{R}$ is continuous and of bounded variation in the sense of Arzel\'{a}, then $\dim_B(G(f))=\dim_H (G(f)) =2.$
\end{lemma}
\begin{proof}
Since $f :[a,b] \times [c,d] \rightarrow \mathbb{R}$ is continuous $\dim_H(G(f)) \ge
     \dim_H \big( [a,b] \times [c,d] \big)=2$.  Thus $$2 \le \dim_H(G(f)) \le  \underline{\dim}_B(G(f))\le  \overline{\dim}_B(G(f)).$$
\par Consider a square net, that is, a set of parallels to the coordinate axes $x=x_i$ and $y=y_j$ such that $x_{i+1}-x_i = y_{j+1}-y_j=D$ is a constant independent of $i,j$ that covers the whole plane. A finite number of squares so determined  will contain points of $[a,b] \times [c,d]$. Denote by $\omega_r$ the oscillation of $f(x,y)$ in the $r$-th square.  Since $f :[a,b] \times [c,d] \rightarrow \mathbb{R}$ is of bounded variation in the sense of Arzel\'{a}, it follows  that \cite{James} $\sum_{r} D w_r$ is bounded for all such nets in which $D$ is less than a fixed number.
\par
Let $ 0 < \delta < \min\{b-a,d-c \}$, $ \frac{b-a}{\delta} \le m \le 1+ \frac{b-a}{\delta }$ and $ \frac{d-c}{\delta} \le n \le 1+ \frac{d-c}{\delta }$ for some $m ,n \in \mathbb{N}.$ From Lemma \ref{Falbiv} the number of $\delta-$cubes that intersect $G(f)$ is
$$N_{\delta}(G(f)) \le 2mn + \frac{1}{\delta} \sum_{j=1}^{n} \sum_{i=1}^{m} R_f[A_{ij}].$$
 From the previous paragraph it follows that  $ \sum_{j=1}^{n} \sum_{i=1}^{m}  \delta R_f[A_{ij}]$ is bounded for all $\delta$ where $ 0< \delta < \delta_0$ for some fixed $\delta_0 >0.$
 To calculate the box dimension of $G(f)$, it suffices to work with  small $\delta -$cover of $G(f)$ and hence assume that $ \sum_{j=1}^{n} \sum_{i=1}^{m} \delta R_f[A_{ij}]$ is bounded for all sufficiently small $\delta >0.$ That is, there exists a constant  $K >0$ such that $$\sum_{j=1}^{n} \sum_{i=1}^{m} R_f[A_{ij}] =  \sum_{j=1}^{n} \sum_{i=1}^{m} \omega_{ij} \le K.\frac{1}{\delta}$$ for sufficiently small $\delta>0$.
 Consequently,
$$ \overline{\lim}_{\delta \rightarrow 0} \frac{\log N_{\delta}(G(f))}{-\log \delta} \le \lim_{\delta \rightarrow 0}\frac{\log(2mn + \frac{1}{\delta}K\frac{1}{\delta})}{-\log \delta},$$
which on calculation produces $$ \overline{\dim}_B(G(f))= \overline{\lim}_{\delta \rightarrow 0} \frac{\log N_{\delta}(G_f)}{-\log \delta} \le 2 ,$$
completing the proof.
\end{proof}
The upcoming theorem which provides the box dimension and Hausdorff dimension of the graph of the mixed Katugampola integral can be deduced as a direct consequence of the previous lemma, Theorem \ref{BVK_Conti} and Theorem \ref{BVK_boundedvar}.
\begin{theorem}\label{BVK_main2}
Let   $ p,q >-1$ and $\alpha, \beta >0.$ Suppose that $f$ is a continuous function of bounded variation on $[a,b] \times [c,d]$. Then $$\dim_B\Big(G\big(^{(p,q)}_{(a,c)} \mathfrak{I}^{(\alpha, \beta)}f\big)\Big)=\dim_H\Big(G\big(^{(p,q)}_{(a,c)} \mathfrak{I}^{(\alpha, \beta)}f\big)\Big)=2.$$
\end{theorem}
\begin{remark}
 Let $g:[a,b] \rightarrow \mathbb{R}$ be continuous. We define a bivariate function $f :[a,b] \times [c,d] \rightarrow \mathbb{R}$ by $f(x,y)=g(x).$ Clearly,  $f$ is continuous on $[a,b] \times [c,d].$ We have 
\begin{equation*}
\begin{aligned} 
\big(^{(p,q)}_{(a,c)} \mathfrak{I}^{(\alpha, \beta)}f\big)(x,y)= & \frac{(p+1)^{1-\alpha}(q+1)^{1-\beta}}{\Gamma (\alpha) \Gamma (\beta )} \int_a ^x \int_c ^y (x^{p+1}-s^{p+1})^{\alpha-1}  \\ & .(y^{q+1}-t^{q+1})^{\beta-1}s^{p} t^{q}f(s,t)~\mathrm{d}s~\mathrm{d}t.
\end{aligned}
\end{equation*}
For $\beta=1$ and $q=0$ we obtain $$  \big(^{(p,0)}_{(a,c)} \mathfrak{I}^{(\alpha, 1)}f\big)(x,y)= \frac{(p+1)^{1-\alpha}}{\Gamma ( \alpha ) } \int_a ^x \int_c ^y (x^{p+1}-s^{p+1})^{\alpha-1}  s^{p} f(s,t)~\mathrm{d}s~\mathrm{d}t .$$
\begin{equation*}
 \begin{aligned}
 \big(^{(p,0)}_{(a,c)} \mathfrak{I}^{(\alpha, 1)}f\big)(x,y)=&~ \frac{(p+1)^{1-\alpha}}{\Gamma ( \alpha ) } \int_a ^x \int_c ^y (x^{p+1}-s^{p+1})^{\alpha-1}  s^{p} f(s,t)~\mathrm{d}s~\mathrm{d}t\\
 =&~ (y-c)\frac{(p+1)^{1-\alpha}}{\Gamma (\alpha ) } \int_a ^x  (x^{p+1}-s^{p+1})^{\alpha-1}  s^{p} g(s)~\mathrm{d}s\\= & ~ (y-c)\frac{(p+1)^{1-\alpha}}{\Gamma (\alpha ) } \int_a ^x  (x^{p+1}-s^{p+1})^{\alpha-1}  s^{p} g(s)~\mathrm{d}s\\
 =&~(y-c) \big(^{p}_a \mathfrak{I}^{\alpha}g\big)(x),
\end{aligned}
 \end{equation*}
 where $^{p}_a \mathfrak{I}^{\alpha}g$ is the standard Katugampola fractional integral of the univariate function $g$.
 By the previous relation between the mixed Katugampola fractional integral and Katugampola fractional integral we have

 $$  \dim_B \big(G\big(^{(p,0)}_{(a,c)} \mathfrak{I}^{(\alpha, 1)}f\big) \big) \le \dim_B \big(G( ^{p}_a \mathfrak{I}^{\alpha}g) \big)+1, $$
whenever $g$ is continuous. If $g$ is a continuous function  of bounded variation on $[a,b]$, by \textup{\cite[Theorem 3.8]{SV}}
$$  \dim_B \big(G( ^{p}_a \mathfrak{I}^{\alpha}g) \big)=1.  $$
Consequently, $\dim_B \big(G\big(^{(p,0)}_{(a,c)} \mathfrak{I}^{(\alpha, 1)}f\big) \big)=2$. This corroborates the previous theorem that
$\dim_B \big(G\big(^{(p,0)}_{(a,c)} \mathfrak{I}^{(\alpha, 1)}f\big) \big)=2$.
\end{remark}
The following theorem  provides an upper bound for the upper box dimension of the graph of the mixed Katugampola fractional integral of a continuous function.
\begin{theorem}\label{Holder_bivariate}
For $0 < a < b < \infty,$ $-1 < p,q \le 0$ and $0 < \alpha, \beta < 1.$
If $f : [a,b]\times [c,d] \to \mathbb{R}$ is continuous, then the upper box dimension of the graph of the mixed Katugampola fractional integral of $f$ is at most $ 3 - \min\{\alpha,\beta\}.$
\end{theorem}
\begin{proof}
Let $ 0<a\le x <x+h\le b $ and $ 0<c\le y <y+k\le d $.  Note that
\begin{equation*}
    \begin{aligned}
     &   \big(^{(p,q)}_{(a,c)} \mathfrak{I}^{(\alpha, \beta)}f \big)      (x+h, y+k)-~  \big(^{(p,q)}_{(a,c)} \mathfrak{I}^{(\alpha, \beta)}f \big) (x,y)\\=& \frac{(p+1)^{1-\alpha}(q+1)^{1-\beta}}{\Gamma (\alpha) \Gamma (\beta)} \Bigg[\int_a ^{x+h}\int_c ^{y+k} ((x+h)^{p+1}-s^{p+1})^{\alpha-1} ((y+k)^{q+1}-t^{q+1})^{\beta-1}\\&. s^{p}t^{q}f(s,t)~\mathrm{d}s~\mathrm{d}t-
     \int_a ^x \int_c ^y (x^{p+1}-s^{p+1})^{\alpha-1}(y^{q+1}-t^{q+1})^{\beta-1} s^{p} t^{q}f(s,t)~\mathrm{d}s~\mathrm{d}t \Bigg]\\
      =&~ J_1 +J_2+J_3+J_4,
      \end{aligned}
               \end{equation*}
               wherein
               \small
      \begin{equation*}
    \begin{aligned}
      J_1=&~ \frac{(p+1)^{1-\alpha}(q+1)^{1-\beta}}{\Gamma (\alpha) \Gamma (\beta)} \int_a ^x \int_c ^y \Big[((x+h)^{p+1}-s^{p+1})^{\alpha-1} ((y+k)^{q+1}-t^{q+1})^{\beta-1}\\&- (x^{p+1}-s^{p+1})^{\alpha-1}(y^{q+1}-t^{q+1})^{\beta-1}\Big] s^{p} t^{q}f(s,t)~\mathrm{d}s~\mathrm{d}t,\\
        J_2=&~ \frac{(p+1)^{1-\alpha}(q+1)^{1-\beta}}{\Gamma (\alpha) \Gamma (\beta)} \int_a ^x \int_y ^{y+k} ((x+h)^{p+1}-s^{p+1})^{\alpha-1} ((y+k)^{q+1}-t^{q+1})^{\beta-1} \\&.s^{p} t^{q}f(s,t)~\mathrm{d}s~\mathrm{d}t,\\
        J_3=&~ \frac{(p+1)^{1-\alpha}(q+1)^{1-\beta}}{\Gamma (\alpha) \Gamma (\beta)} \int_x ^{x+h} \int_c ^y ((x+h)^{p+1}-s^{p+1})^{\alpha-1} ((y+k)^{q+1}-t^{q+1})^{\beta-1}\\&. s^{p} t^{q}f(s,t)~\mathrm{d}s~\mathrm{d}t,\\
        J_4=&~ \frac{(p+1)^{1-\alpha}(q+1)^{1-\beta}}{\Gamma (\alpha) \Gamma (\beta)} \int_x ^{x+h} \int_y ^{y+k} ((x+h)^{p+1}-s^{p+1})^{\alpha-1} ((y+k)^{q+1}-t^{q+1})^{\beta-1}\\&. s^{p} t^{q}f(s,t)~\mathrm{d}s~\mathrm{d}t.\\
            \end{aligned}
               \end{equation*}
               \normalsize
Since $f$ is continuous on $[a,b]\times [c,d],$ there exists  $M>0$ such that $$|f(t,u)| \le M ~~ \forall~ (t,u) \in [a,b] \times [c,d].$$ Next we shall bound  $J_1$  as follows:
\begin{equation*}
\begin{aligned}
      |J_1| \le&  \frac{M(p+1)^{1-\alpha}(q+1)^{1-\beta}}{\Gamma (\alpha) \Gamma (\beta)} \int_a ^x \int_c ^y \Big[ (x^{p+1}-s^{p+1})^{\alpha-1}(y^{q+1}-t^{q+1})^{\beta-1}\\&~-((x+h)^{p+1}-s^{p+1})^{\alpha-1} ((y+k)^{q+1}-t^{q+1})^{\beta-1}\Big] s^{p} t^{q}~\mathrm{d}s~\mathrm{d}t.
      \end{aligned}
               \end{equation*}
Considering the transformations $s^{p+1}=u$ and $t^{q+1}=v$ in the above integral   we obtain
\small
\begin{equation*}
\begin{aligned}
      |J_1| \le ~& \frac{M(p+1)^{-\alpha}(q+1)^{-\beta}}{\Gamma (\alpha) \Gamma (\beta)} \int_{a^{p+1}} ^{x^{p+1}}\int_{c^{q+1}} ^{y^{q+1}} \Big[(x^{p+1}-u)^{\alpha-1}(y^{q+1}-v)^{\beta-1}\\~&- ((x+h)^{p+1}-u)^{\alpha-1} ((y+k)^{q+1}-v)^{\beta-1} \Big] ~\mathrm{d}u~\mathrm{d}v\\
=~& \frac{M(p+1)^{-\alpha}(q+1)^{-\beta}}{\Gamma (\alpha) \Gamma (\beta)}   \int_{a^{p+1}} ^{x^{p+1}}\int_{c^{q+1}} ^{y^{q+1}}\Big[(x^{p+1}-u)^{\alpha-1}(y^{q+1}-v)^{\beta-1} \\ ~& - ((x+h)^{p+1}-u)^{\alpha-1} (y^{q+1}-v)^{\beta-1} + ((x+h)^{p+1}-u)^{\alpha-1} (y^{q+1}-v)^{\beta-1}\\&~- ((x+h)^{p+1}-u)^{\alpha-1} ((y+k)^{q+1}-v)^{\beta-1} \Big] ~\mathrm{d}u~\mathrm{d}v\\
=~& \frac{M(p+1)^{-\alpha}(q+1)^{-\beta}}{\Gamma (\alpha) \Gamma (\beta)} \Bigg[ \int_{a^{p+1}} ^{x^{p+1}}\int_{c^{q+1}} ^{y^{q+1}}(y^{q+1}-v)^{\beta-1}\Big[(x^{p+1}-u)^{\alpha-1} \\ & ~ - ((x+h)^{p+1}-u)^{\alpha-1}\Big]~\mathrm{d}u~\mathrm{d}v  + \int_{a^{p+1}} ^{x^{p+1}}\int_{c^{q+1}} ^{y^{q+1}}
((x+h)^{p+1}-u)^{\alpha-1} \\&~ .\Big[  (y^{q+1}-v)^{\beta-1}- ((y+k)^{q+1}-v)^{\beta-1}\Big] ~\mathrm{d}u~\mathrm{d}v   \Bigg].
%=~ & \frac{M}{(p+1)^{\alpha} (q+1)^{\beta}\Gamma (\alpha+1)\Gamma (\beta+1)}  ~~\Bigg[\Big[\big((x+h)^{p+1}-x^{p+1}\big)^{\alpha}-\big((x+h)^{p+1}-a^{p+1}\big)^{\alpha}\Big]\\ & \Big[\big((y+k)^{q+1}-y^{q+1}\big)^{\beta}-\big((y+k)^{q+1}-c^{q+1}\big)^{\beta}\Big]\\&+(x^{p+1}-a^{p+1})^{\alpha}(y^{q+1}-c^{q+1})^{\beta}\Bigg].
\end{aligned}
               \end{equation*}
               \normalsize
  Using the Bernoulli's inequality that reads $(1+x)^r \le 1+rx$ for $0 \le r\le 1$ and $x \ge-1$   we get
\begin{equation*}
\begin{aligned}
  K_1= &~   \int_{a^{p+1}} ^{x^{p+1}}(x^{p+1}-u)^{\alpha-1}- ((x+h)^{p+1}-u)^{\alpha-1}~\mathrm{d}u \\
  = &~ \frac{1}{\alpha} \Big [  \big((x+h)^{p+1} -x^{p+1}\big)^\alpha - \big((x+h)^{p+1} -a^{p+1}\big)^\alpha + (x^{p+1}-a^{p+1})^\alpha  \Big] \\
   \le &~ \frac{(p+1)^\alpha a^{\alpha p} h^\alpha} {\alpha}.
\end{aligned}
               \end{equation*}
               Similarly, we have

    \begin{equation*}
\begin{aligned}
  K_2= &~   \int_{c^{q+1}} ^{y^{q+1}}(y^{q+1}-v)^{\beta-1}- ((y+k)^{q+1}-v)^{\beta-1} ~\mathrm{d}v\\
   \le &~ \frac{(q+1)^\beta c^{\beta q} k^\beta} {\beta}.
\end{aligned}
               \end{equation*}
 Using the estimate for $K_1$ and $K_2$, $|J_1|$ may be bounded as follows:
 \begin{equation*}
\begin{aligned}
      |J_1| \le ~& \frac{M(p+1)^{-\alpha}(q+1)^{-\beta}}{\Gamma (\alpha) \Gamma (\beta)}\Bigg[\frac{(p+1)^\alpha a^{\alpha p} h^\alpha} {\alpha} \int_{c^{q+1}} ^{y^{q+1}} (y^{q+1}-v)^{\beta-1}~\mathrm{d}v \\& ~+  \frac{(q+1)^\beta c^{\beta q} k^\beta} {\beta} \int_{a^{p+1}} ^{x^{p+1}} ((x+h)^{p+1}-u)^{\alpha-1} ~\mathrm{d}u \Bigg] \\
      \le &~   \frac{M(p+1)^{-\alpha}(q+1)^{-\beta}}{\Gamma (\alpha) \Gamma (\beta)}\Bigg[\frac{(p+1)^\alpha a^{\alpha p} h^\alpha} {\alpha \beta} (d^{q+1}-c^{q+1})^\beta \\ &~+ \frac{(q+1)^\beta c^{\beta q} k^\beta} {\alpha \beta} (b^{p+1} -a^{p+1})^\alpha\Bigg]
        \end{aligned}
               \end{equation*}
Therefore, for a suitable constant $C_1$ we have
$$|J_1| \le C_1 (h^\alpha + k^\beta).$$
Similarly, with suitable constants $C_2$, $C_3$ and $C_4$, a now
familiar argument lead to  the following.

     $$ |J_2| \le  C_2 k^\beta,\quad
      |J_3| \le  C_3 h^\alpha,\quad
      |J_4| \le C_4 h^\alpha k^\beta.$$
      Consequently,  we see that for sufficiently small positive constants $h,k
      $, $\gamma:=\min\{\alpha,\beta\}$  and a suitable constant $C$
\begin{equation*}
\begin{aligned}
     \Big|  \big(^{(p,q)}_{(a,c)} \mathfrak{I}^{(\alpha, \beta)}f \big)    (x+h,y+k)- ~\big(^{(p,q)}_{(a,c)} \mathfrak{I}^{(\alpha, \beta)}f \big) (x,y)\Big| \le &~ |J_1|+
 |J_2|+|J_3| +|J_4| \\ \le& ~ C (h^{\alpha}+k^{\beta})\\
  \le&~ C \Big(h^ \gamma +k^\gamma \Big)\\
  \le &~ C \| (x+h,y+k) -(x,y)\|_2^\gamma.
\end{aligned}
               \end{equation*}
 In view of Corollary \ref{BVK_holder_dim} the above estimate completes the proof.
\end{proof}

We prove the semigroup property of the mixed Katugampola integral introduced in the present note. This may be treated as a bivariate counterpart of \textup{\cite[Theorem 4.1]{Katu}}.
\begin{theorem}\label{semigroup} For an integrable function $f:[a,b]\times [c,d] \rightarrow \mathbb{R}$ for which the mixed Katugampola fractional integral  $^{(p,q)}_{(a,c)} \mathfrak{I}^{(\alpha, \beta)}f $ exists, the semigroup property holds. That is,

$$^{(p,q)}_{(a,c)} \mathfrak{I}^{(\alpha, \beta)}~ ^{(p,q)}_{(a,c)} \mathfrak{I}^{(\gamma, \theta)} f = ^{(p,q)}_{(a,c)} \mathfrak{I}^{(\alpha+\gamma, \beta+\theta)}f.$$
\end{theorem}
\begin{proof}
Using the Fubini's theorem and the Dirichlet technique, we obtain
\begin{equation*}
 \begin{aligned}
& \big(^{(p,q)}_{(a,c)} \mathfrak{I}^{(\alpha, \beta)}~ ^{(p,q)}_{(a,c)} \mathfrak{I}^{(\gamma, \theta)} f\big) (x,y) \\& = \frac{(p+1)^{2-\alpha-\gamma}(q+1)^{2-\beta-\theta}}{\Gamma (\alpha) \Gamma (\beta) \Gamma (\gamma) \Gamma (\theta)} \int_a ^x \int_c ^y \Bigg[\int_s ^x \int_t ^y (x^{p+1}-u^{p+1})^{\alpha-1} (u^{p+1}-s^{p+1})^{\gamma-1}\\& .(y^{q+1}-v^{q+1})^{\beta-1} (v^{q+1}-t^{q+1})^{\theta-1}u^{p} v^{q} ~\mathrm{d}u ~ \mathrm{d}v \Bigg]~ s^{p}t^{q}f(s,t)~\mathrm{d}s ~ \mathrm{d}t.
\end{aligned}
\end{equation*}
Using the change of variable $z=\dfrac{u^{p+1}-s^{p+1}}{x^{p+1}-u^{p+1}}$ and the well known formula for beta function we have
\begin{equation*}
 \begin{aligned}
  \int_s ^x  (x^{p+1}-u^{p+1})^{\alpha-1} (u^{p+1}-s^{p+1})^{\gamma-1}u^{p}  ~\mathrm{d}u=& \frac{(x^{p+1}-s^{p+1})^{\alpha+\gamma-1}}{p+1} \int_0^1 (1-z)^{\alpha-1} z^{\gamma-1} ~\mathrm{d}z\\=& \frac{(x^{p+1}-s^{p+1})^{\alpha+\gamma-1}}{p+1} \frac{\Gamma (\alpha)\Gamma (\gamma)}{\Gamma (\alpha+\gamma)}.
  \end{aligned}
\end{equation*}
Consequently
\begin{equation*}
 \begin{aligned}
\big(^{(p,q)}_{(a,c)} \mathfrak{I}^{(\alpha, \beta)}~ ^{(p,q)}_{(a,c)} \mathfrak{I}^{(\gamma, \theta)} f\big) (x,y) =& \frac{(p+1)^{1-(\alpha+\gamma)}(q+1)^{1-(\beta+\theta)}}{\Gamma (\alpha+\gamma) \Gamma (\beta+\theta) } \int_a ^x \int_c ^y  (x^{p+1}-s^{p+1})^{\alpha+\gamma-1}\\&.  (y^{q+1}-t^{q+1})^{\beta+\theta-1} s^{p}t^{q}f(s,t)~\mathrm{d}s ~ \mathrm{d}t\\ =&  \big(^{(p,q)}_{(a,c)} \mathfrak{I}^{(\alpha+\gamma, \beta+\theta)}f\big)~(x,y),
\end{aligned}
\end{equation*}
and the proof is complete.
\end{proof}
 With the help of the relation between the Hausdorff dimension and the upper box dimension, Theorem \ref{Holder_bivariate} and Theorem \ref{semigroup},
 one can easily prove:
\begin{theorem}
Let $f: [a,b] \times [c,d] \to \mathbb{R}$ be continuous,  $ 0 < a < b < \infty $, $0<c<d<\infty$ and $-1< p,q \le 0.$
\begin{enumerate}
\item  If $ 0< \alpha, \beta < 1,$ then $$ 2 \le  \dim_H \Big(G({^{(p,q)}_{(a,c)} \mathfrak{I}^{(\alpha, \beta)}f}) \Big)
 \le  \dim_B \Big(G({^{(p,q)}_{(a,c)} \mathfrak{I}^{(\alpha, \beta)}f}) \Big)   \le 3- \min\{\alpha,\beta\}.$$
 \item If $  \alpha , \beta \ge 1,$ then $$  \dim_H \Big(G({^{(p,q)}_{(a,c)} \mathfrak{I}^{(\alpha, \beta)}f}) \Big)
 =  \dim_B \Big(G({^{(p,q)}_{(a,c)} \mathfrak{I}^{(\alpha, \beta)}f}) \Big) =2 .$$
\end{enumerate}
\end{theorem}

\section{Some Examples}
The fact  that the fractal dimension of the graph of a bivariate continuous function of bounded variation (in the sense of Arzel\'{a}) is $2$ may motivate to ask whether there exists a bivariate continuous function which is not of bounded variation with box dimension of its graph being $2$. It is to this that we now turn. A  univariate counterpart of the construction herein can be consulted in \cite{SV}.
 \par
 Given $[a,b] \times [c,d]$, consider a monotonically  increasing sequence of real numbers $(a_n)$ in $[a,b]$ such that $\lim_{n\to \infty} a_n=b$. For example, we consider a sequence $(a_n)_{n \ge 0}$ by defining $a_0=a$ and  $a_n=a+\frac{b-a}{2}+\dots+\frac{b-a}{2^n}$, $n \in \mathbb{N}$.
Let $\phi: [a_0,a_1] \times [c,d]$ be a continuous function  such that $$\phi(a_0,y)=\phi(a_1,y)~ \forall~ y \in [c,d].$$ Let us call $\phi$ as the generating function.
Let $\psi_n$ be map from $[a_{n-1},a_n]$ onto $[a_0,a_1]$ defined by
$$\psi_n(x) = \frac{2^n \big[(a_1-a_0)x+a_0a_n-a_1 a_{n-1} \big] }{b-a}.$$
Let $F_1(x,y)=\phi(x,y)$ for $(x,y) \in [a_0,a_1] \times [c,d]$ and for $n \ge 2$,  $$F_n(x,y)=\frac{1}{n}\phi\big( \psi_n(x),y\big)+\frac{n-1}{n} \phi(a_0,y)~ \text{for}~ x \in [a_{n-1},a_n] \times [c,d]. $$
 Consider a sequence of functions $(T_n)_{n \in \mathbb{N}}$ on $[a,b]$ as follows. For each $n \in \mathbb{N},$
 \begin{equation*}
 T_n(x,y)= \begin{cases}
   F_k (x,y) ~\text{for}~(x,y) \in [a_{k-1},a_k]\times [c,d], ~k=1,2,\dots,n\\
      F_n(a_n,y) ~\text{for}~(x,y) \in [a_n,b]\times [c,d].
 \end{cases}
 \end{equation*}
Let   $$T(x,y)=\lim_{n\rightarrow \infty} T_n(x,y).$$
The following theorem is obtained by modifying and adapting \textup{\cite[Theorem 4.1]{SV}} to the present setting of bivariate function.
\begin{theorem}\label{constructhm1}
If the generating function  $\phi$ is a non-constant function along the line $y=y_0$ for some $y_0 \in [c,d],$ then the function $T$ is not of bounded variation on $[a,b] \times [c,d].$
\end{theorem}
\begin{proof}
Since $\phi$ is a non-constant function along the line $y=y_0$ for some $y_0 \in [c,d],$  we have $$|\phi(t,y_0)-\phi(u,y_0)| \ge c,$$  for some $t,u \in [a_0,a_1] $ with $t \le u, $ and $c>0$. Choose $ t_1, t_2 \in [a_0,a_1]$, $t_1 <t_2$ such that $$|F_1(t_1,y_0)-F_1(t_2,y_0)|=|\phi(t,y_0)-\phi(u,y_0)| \ge c.$$ We can choose
$ t_3, t_4 \in [a_1,a_2]$ such that $t_3 < t_4$ and $$|F_2(t_3,y_0)-F_2(t_4,y_0)|= \frac{1}{2} |\phi(t,y_0)-\phi(u,y_0)| \ge \frac{c}{2}.$$ Using the linear maps we see that for $i >1$,  $t_i= \psi_{i-1}^{-1}(t)$ and $t_{i+1}= \psi_{i-1}^{-1}(u)$. Continuing like this, we obtain a collection $P'=\{t_i:t_1 < t_2< t_3 < \dots < t_{2n} \}.$ We take a  partition $P$ of $[a,b],$ which contains $P'.$ The variation of $T$ (treated as a univariate function) along the line $y=y_0$ denoted by $V(T,[a,b],y_0)$ is $$ V(T,[a,b],y_0) \ge \sum_{i=1}^{2n}|T(t_{i+1},y_0)-T(t_i,y_0)| \ge \sum_{i=1}^{n}|F_i (t_{i+1},y_0)-F_i (t_i,y_0)| \ge \sum_{i=1}^{n} \frac{c}{i}.$$ Since $\sum_{i=1}^{\infty} \frac{1}{i} = \infty$ and $c >0,$ restriction of $T$ along the line $y=y_0$ is not of bounded variation on $[a,b].$ From a well-known property of a univariate function of bounded variation on $[a,b]$, we obtain that the restriction of $T$, $T|_{y=y_0},$ along the line $y=y_0$ can not be written as difference of two increasing functions $f_1,f_2:[a,b] \rightarrow \mathbb{R}.$ That is, $T|_{y=y_0}= f_1 -f_2$ with $\triangle_{10}f_i(x,y_0) \ge 0 , i=1,2$ does not hold. This clearly indicates that $T$ will not be of bounded variation on $[a,b] \times [c,d]$ in the sense of Arzel\'a.
\end{proof}

\begin{theorem}\label{constructhm2}
If $ \phi $ is of bounded variation in the sense of Arzel\'a, then
the box dimension of its graph $G(T)$ is $2.$
\end{theorem}
\begin{proof}
Since the function $\phi$ is of bounded variation, surface area of $z= \phi(x,y)$ is finite say $L;$ see \cite{James}. We have
 $$ 2 \le \dim_H(G(T)) \le \underline{\dim}_B(G(T)) \le  \overline{\dim}_B(G(T)).$$ Consequently, it is enough to show that $\overline{\dim}_B(G_T) \le 2.$ Let $ 0< \delta < \min\{b-a,d-c\},$ $m$ and $n$ be the natural numbers such that $$\frac{b-a}{ \delta} \le m < 1+ \frac{b-a}{ \delta}, \quad \frac{d-c}{ \delta} \le n < 1+ \frac{d-c}{ \delta}.$$ Denote by $N_{\delta}(G(T))$  the number of squares of the $\delta-$mesh that intersect $G(T). $ By Lemma \ref{Falbiv} $$N_{\delta}(G(T)) \le 2mn+L. \delta^{-2} \sum_{i=1}^{ \left \lceil {\frac{m}{b-a}}\right \rceil} \frac{1}{i},$$ where $\left \lceil {.}\right \rceil $ denotes the ceiling function. Using the inequality $ \sum_{i=1}^{n} \frac{1}{i} \le \log n +1 $, we obtain
 $$  \frac{\log N_{\delta}(G(T))}{- \log \delta} \le \frac{\log\Big(2mn+L. \delta^{-2}\big( \log (\left \lceil {\frac{m}{b-a}}\right \rceil )+1 \big) \Big)}{- \log \delta}.$$
 This gives, $ \overline{\lim}_{m\rightarrow \infty}\frac{\log N_{\delta}(G(T))}{- \log \delta} \le 2$ and consequently, $\overline{\dim}_B(G(T)) \le 2$ providing the assertion.
\end{proof}
\begin{theorem}
If $ \phi $ is of bounded variation in the sense of Arzel\'a, then
the Hausdorff dimension of $G(T)$ is $2.$
\end{theorem}
\begin{proof}
By hypotheses that function $\phi$ is of bounded variation in the sense of Arzel\'a, we have $\dim_H \big(G(\phi)\big)=2 .$
Note that $\dim_H \big(G(F_n)\big)=2 ~ ~\forall~ n \in \mathbb{N}.$
By the countable additivity of Hausdorff dimension, the result follows.
\end{proof}

\begin{theorem}{(Theorem 16, \cite{Raymond})}
If $f:[a,b]\times [c,d] \rightarrow \mathbb{R}$ is of bounded variation (in the sense of Arzel\'a), then $f$ is totally differentiable almost everywhere.
\end{theorem}
The proof of the upcoming theorem follows from the above theorem and the construction of $T$, and  hence omitted.
\begin{theorem}
If the generating function $\phi$ is of bounded variation in the sense of Arzel\'a, then $T$ is totally differentiable almost everywhere.
\end{theorem}
Theorems \ref{constructhm1} and \ref{constructhm2} provide examples for a bivariate continuous function that is not of bounded variation (in the sense of
Arzel\'a) with graph having box dimension $2$. Let us consider the region $[0,1] \times [0,1]$, $a_0=0, a_1=0.5 , \dots $ and the generating function defined on $[0, 0.5] \times [0,1]$ as
$\phi(x,y) =  x (x-0.5) \sin y $, $ \phi(x,y)= \sin \big(x(x-0.5)  \big)$. Correspondingly, we obtain the bivariate continuous functions that  are not of bounded variations with graphs having fractal dimension $2$ depicted in the following figures.

\begin{figure}[h!]
\centering
\epsfig{file=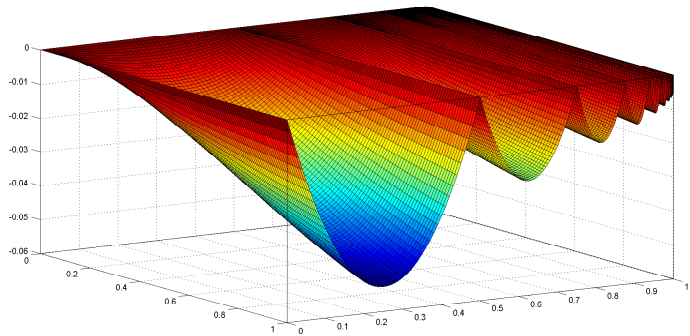,scale=0.3}
 \caption{Continuous function $T$ of unbounded variation with generating function $\phi(x,y) =  x (x-0.5) \sin y $.}
 \end{figure}
 \begin{figure}[h!]
\centering
 \epsfig{file=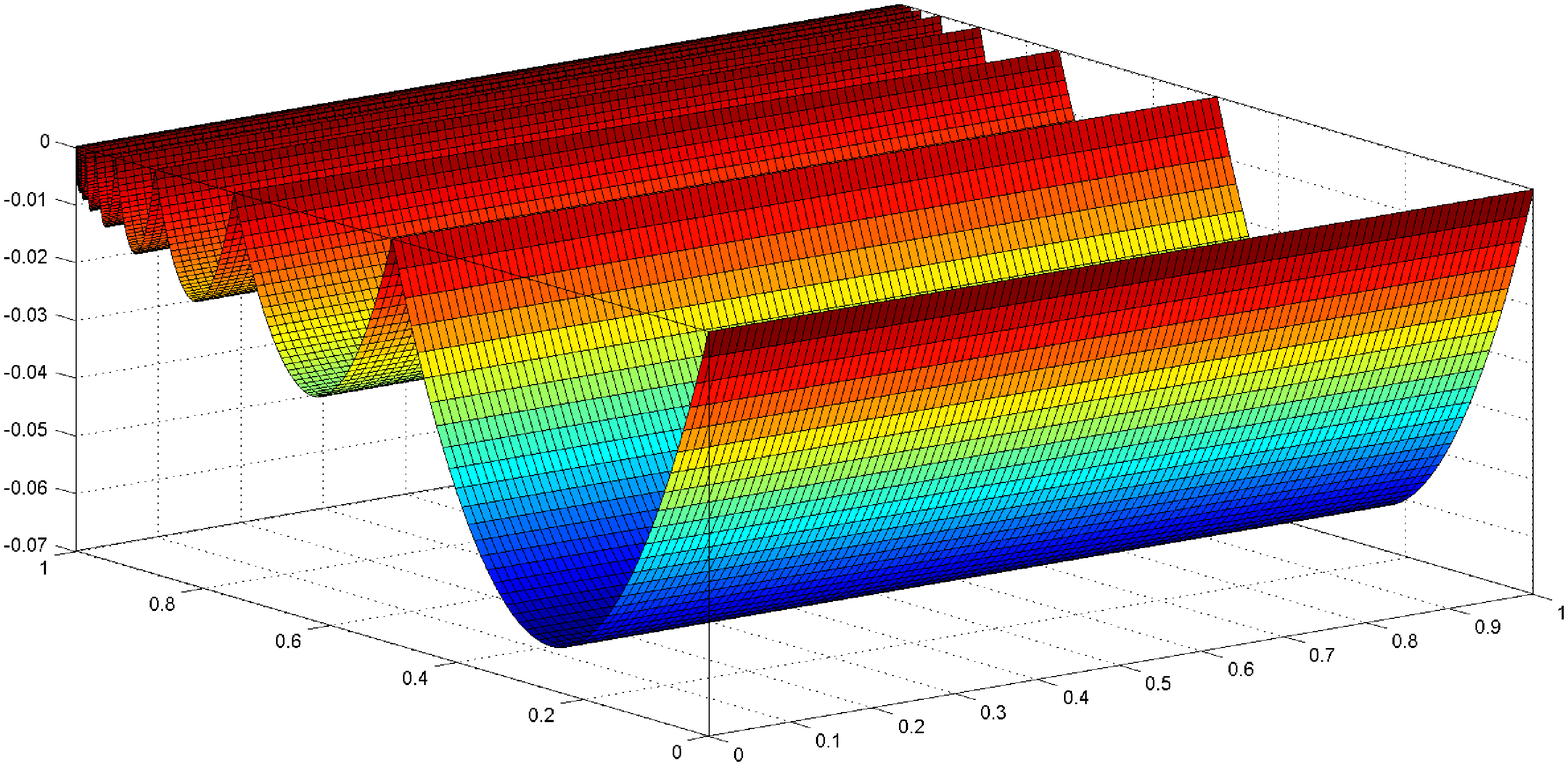,scale=0.3}
 \caption{Continuous function $T$ of unbounded variation with generating function $ \phi(x,y)= \sin \big(x(x-0.5)  \big)$}
 \end{figure}

Next we present a simple example of a function which is neither continuous nor of bounded variation, with its graph
having the box dimension and Hausdorff dimension $2.$
\begin{example}
Let $A=[a,b] \times [c,d]$ and define a function $f:A \rightarrow \mathbb{R}$ as follows.
 \begin{equation*}
 f(x,y) =
 \begin{cases}
 0 \quad \text{if} \quad x, y ~~\text{are rational~ numbers,}\\
 1 \quad \text{otherwise}.
 \end{cases}
 \end{equation*}
 We may write the graph of the function $f$ as
 \begin{equation*}
 \begin{split}
 G(f)=&~ \big\{(x,y,0): (x,y) \in A \cap (\mathbb{Q} \times \mathbb{Q}) \big\} \cup \big\{(x,y,1): (x,y) \in A  \cap (\mathbb{Q} \times \mathbb{Q})^c \big\}\\ := &~G_1 \cup G_2
 \end{split}
 \end{equation*}
 The set $G_1$ is countable, so the Hausdorff measure is zero \cite{Fal}. Using the countable additive property of Hausdorff measure \cite{Fal}, we have $$\dim_H (G(f)) = 2.$$ Therefore $2 = \dim_H (G(f)) \le \underline{\dim}_B (G(f)) \le \overline{\dim}_B (G(f)).$ Using a property of the upper box dimension \cite{Fal}, we have  $$\overline{\dim}_B(G_1)=\overline{\dim}_B( \overline{G_1})=2, \quad \overline{\dim}_B(G_2)=\overline{\dim}_B( \overline{G_2})=2.$$ Since $\overline{\dim}_B$ is finitely stable, we get $$\overline{\dim}_B (G(f))=2.$$ Therefore, $\dim_B (G(f))=\dim_H (G(f))=2.$
\end{example}

\subsection*{Acknowledgements}
 The first author thanks the University Grants
  Commission (UGC), India for financial support in the form of a Senior Research Fellowship.

\bibliographystyle{amsplain}

\begin{thebibliography}{10}
	
	
	 
	
	
	
\bibitem {Raymond} C. R. Adams, J. A. Clarkson, \textit{Properties of functions $f(x,y)$ of bounded variation,} Trans. Amer. Math. Soc., 36 (1934) 711-730.
\bibitem {James} J. A. Clarkson, C. R. Adams, \textit{On definitions of bounded variation for functions of two variables,} Trans. Amer. Math. Soc., 35 (1933) 824-854.
    \bibitem {Fal} K. J. Falconer, \textit{Fractal Geometry: Mathematical Foundations and Applications,} John Wiley Sons Inc., New York, 1999.
\bibitem{Katu} U. N. Katugampola, \textit{New approach to a generalized fractional integral,}
Appl. Math. Comput., 218 (2011) 860-865.

\bibitem{KST} A. A. Kilbas, H. M. Srivastava, J. J. Trujillo, \textit{Theory and applications of fractional differential equations,} North-Holland Mathematics Studies, 204,
Elsevier Science B.V., Amsterdam, 2006.     



\bibitem {Liang} Y. S. Liang, \textit{Box dimensions of Riemann-Liouville fractional integrals of continuous functions of bounded variation,} Nonlin. Anal., 72 (2010) 4304-4306.
    
  \bibitem {LS} Y. S. Liang,  W.Y. Su, \textit{Fractal dimensions of fractional integral of continuous functions,} Acta Math. Sin. (Engl. Ser.), 32(12) (2016) 1494-1508. 
      
    \bibitem{YLS}  K. Yao, Y.S. Liang, F.  Zhang, \textit{On the connection between the order of the fractional derivative and the Hausdorff dimension of a fractal function,} Chaos, Solitons \& Fractals,  41(5) (2009) 2538-2545.
      
\bibitem{OM} E. C. Oliveira, J. A. Machado, \textit{A Review of definitions for fractional derivatives and integral,} Math. Problems in Eng., 2014 (2014) Article ID 238459, 6 pages.
    
    \bibitem{RSY} H-J Ruan, W-Y Su, K. Yao, \textit{Box dimension and fractional integral of linear fractal interpolation functions,}  J. Approx. Theory, 161(1) (2009) 187-197.
    
\bibitem {Samko} S. G. Samko, A. A. Kilbas, O. N. Marichev, \textit{Fractional Integrals and Derivatives: Theory and Applications,} Gordon and Breach Science Publishers, 1993.

\bibitem{SVthesis} S. Verma, Some Results on Fractal Functions, Fractal Dimensions and Fractional Calculus, Ph.D. thesis, Indian Institute of Technology Delhi, India, 2020.  


    \bibitem{SV} S. Verma, P. Viswanathan, \textit{A note on Katugampola fractional calculus and fractal dimensions}, Appl. Math. Comp., 339 (2018) 220-230.
 \bibitem{SVarx} S. Verma, P. Viswanathan, \textit{Bivariate functions of bounded variation: Fractal dimension and
fractional integral}, Indag. Math., 31 (2020) 294-309.


 \bibitem{ZW} W.X. Zheng, S.W. Wang, \textit{Real Function and Functional Analysis}, High Education Publication, Beijing, 1980 (in Chinese).
 
 
\end{thebibliography}

\end{document}